\documentclass[12pt,reqno]{amsart}
\usepackage{soul}
\usepackage{multicol,amsmath,graphics, mathtools, cancel,soul,color,bbm,amsfonts,dsfont,tikz,mathrsfs,amssymb,bm,hyperref,comment,xcolor,pgfplots,stmaryrd,amsthm,url,subfigure,comment}

\usepackage[left=1in, right=1in, top=1.1in,bottom=1.1in]{geometry}
\usetikzlibrary{matrix,patterns,positioning,arrows}
\numberwithin{equation}{section}

\hypersetup{
    colorlinks=true, %set true if you want colored links
    linktoc=all,     %set to all if you want both sections and subsections linked
    linkcolor=blue,  %choose some color if you want links to stand out
}

\newcommand{\E}{\mathbb{E}}

\newcommand{\R}{\mathbb{R}}

\newcommand{\vertiii}[1]{{\left\vert\kern-0.25ex\left\vert\kern-0.25ex\left\vert #1
    \right\vert\kern-0.25ex\right\vert\kern-0.25ex\right\vert}}

\def\R{\mathbb{R}}

\def\dh2l{\mathbf{d}_{\mathbb{H}_{2\ell}}}
\def\d2{\mathbf{d}_2}

\newtheorem{theorem}{Theorem}[section]
\newtheorem{proposition}[theorem]{Proposition}
\newtheorem{lemma}[theorem]{Lemma}
\newtheorem{corollary}[theorem]{Corollary}

\theoremstyle{definition}
\newtheorem{definition}[theorem]{Definition}

\theoremstyle{remark}
\newtheorem{remark}[theorem]{Remark}

\def\1{\mathbbm{1}}

\makeatletter
\@namedef{subjclassname@1991}{2020 Mathematics Subject Classification}
\makeatother

\begin{document}
\title{Hockey-Stick Domination and Distributional Comparison on Finite Posets}
\author{Arturo Jaramillo, Sayl\'e Sigarreta}
\address{Arturo Jaramillo: Department of Probability and statistics, Centro de Investigaci\'on en matem\'aticas (CIMAT)}
\email{jagil@cimat.mx}
\address{Sayl\'e Sigarreta: Facultad de Ciencias F\'isico Matem\'aticas, Benem\'erita Universidad Aut\'onoma de Puebla}
\email{sayle.sigarretari@alumno.buap.mx}

\date{\today}


\begin{abstract}
We develop a framework for comparing probability measures on finite posets via hockey-stick domination, an order relation defined through interval-counting test functions. The theory introduces poset integrals, derivatives, power functions and the associated moment functionals, all of which are invariant under poset isomorphisms. We prove that hockey-stick domination admits an exact quantitative characterization: whenever $\mu$ is dominated by $\nu$ in the hockey-stick order, the corresponding Zolotarev-type distance is equal to one half of the second-order poset moment of $\nu-\mu$. We further develop a constructive theory for generating such domination relations. In particular, we show that hockey-stick domination is preserved under direct products, disjoint unions, ordinal sums, and suitable ideal restrictions, yielding natural families of examples on chains, Boolean posets, rectangular lattices, rooted trees, and Young diagrams.

\end{abstract}

%==========================================================================================================================
%Introduction
%==========================================================================================================================

\maketitle

\section{Introduction and motivations}\label{Sec:introandmotivations}
\noindent This paper is devoted to the comparison of probability measures on finite posets through the development of a M\"obius-inversion-based calculus, invariant under poset isomorphisms, combined with a suitably designed notion of domination between measures, which we refer to as hockey-stick domination. Our starting point is the work by Boutsikas-Vaggelatou
\cite{MR1909919}, which links convex order to quantitative discrepancies
between laws via the stop-loss metric and Zolotarev's ideal metric. For
probability measures $\mu,\nu$ on $\R$, set
\begin{align}
d_{\zeta}(\mu,\nu)
  &:=
  \int_{\R}
  \Bigl|
  \int_{\R}(x-t)_{+}\,\mu(dx)
  -
  \int_{\R}(x-t)_{+}\,\nu(dx)
  \Bigr|\,dt.
  \label{eq:solotarevdef}
\end{align}
If $\mu$ is dominated by $\nu$ in convex order, namely if
\begin{align*}
\int_{\R}V(x)(\nu-\mu)(dx)\geq 0
\end{align*}
for every convex function $V:\R\rightarrow\R$, then
\begin{align}
d_{\zeta}(\mu,\nu)
  &\leq
  \frac{1}{2}
  \int_{\R}x^2(\nu-\mu)(dx).
  \label{ineq:dmmunubound}
\end{align}
In particular, for a sequence of probability measures convexly dominated by a fixed target measure, convergence of the variance to that of the target implies convergence in law. This gives a moment comparison principle: under the domination assumption, the difference between the corresponding quadratic moments controls a metric discrepancy. In this sense, the result belongs to a broader family of principles in which convergence of a distinguished moment-type quantity becomes sufficient for convergence in distribution, as in the Fourth Moment Theorem on a fixed Wiener chaos
\cite{NourdinPeccati2009}.\\

\noindent
Relation \eqref{ineq:dmmunubound} naturally raises the question of whether analogous comparison principles can be formulated beyond the ordered real line. In this work, we take a first step in this direction by focusing on finite posets. These provide a convenient class of ordered discrete spaces and arise naturally in many combinatorial settings, including partitions of finite sets, divisors of an integer, flats of a matroid, and partitions of integers (see \cite{stanley2011enumerative,trotter1992combinatorics}). A probability measure on such a poset may be viewed as the law of a random combinatorial object. Although these objects are often studied through real-valued summaries such as sizes, ranks, cardinalities, or numbers of parts, our aim is to compare probability measures directly on the poset, without first pushing them forward to the real line.\\

\noindent
To extend the comparison framework of Boutsikas-Vaggelatou to this setting, one needs substitutes for the affine and differential structure available on $\mathbb{R}$. The main idea of this work is that M\"obius inversion provides a natural replacement for the missing structure. Starting from the incidence algebra of a finite poset, in Section~\ref{sec:calcoverlattices}, we introduce poset integral and derivative operators, both invariant under poset isomorphisms

%================================================================================

\subsection{Main Contributions}
We now present the main contributions of this work, consisting of domination-based metric identities, together with constructive mechanisms for assembling such domination relations in concrete poset settings.\\

\noindent A fundamental ingredient in our construction is the formulation of power functions $x \mapsto \phi_d(x)$ on a poset $\mathcal{P}$, with $d$ in $\mathbb{N}$. These functions are constructed by iterating the poset integral introduced in Section~\ref{sec:calcoverlattices}, applied to the function identically equal to one, and multiplying by the factor $d!$. Once defined, they induce moment functionals of signed measures $\rho$ on $\mathcal P$ via
\begin{align*}
m_d[\rho]
  &:= \int_{\mathcal P} \phi_d(x)\,\rho(dx).
\end{align*}
\noindent
The second ingredient is a family of test functions that plays the role of the real-line hockey-stick functions $(x-t)_{+}$. For $x,z$ in $\mathcal{P}$, we define
\begin{align*}
\Phi_x(z)
:=
\sum_{y\in\mathcal{P}}
\mathbbm{1}_{\{\,x\le y\le z\,\}}.
\end{align*}
Thus, $\Phi_x(z)$ counts the number of elements lying between $x$ and $z$. The family $\{\Phi_x : x \in \mathcal{P}\}$ then induces a comparison preorder on probability measures on $\mathcal P$, introduced in Subsection~\ref{subhs}. We call this relation hockey-stick domination and denote it by $\preceq_{HS}$.\\

\noindent
In Section~\ref{sec:zolotarev}, we define a Zolotarev-type discrepancy $d_\zeta$ adapted to finite posets. With the moment functionals and hockey-stick domination in place,  Theorem~\ref{thm:main}, the first main result of the paper, establishes that under hockey-stick domination this discrepancy admits an exact representation in terms of a second-order poset moment difference.

\begin{theorem}\label{thm:main}
Let $\mu,\nu$ be probability measures on $\mathcal P$. If $\mu$ is dominated by $\nu$ in the hockey-stick order, then
\begin{align*}
d_{\zeta}(\mu,\nu)
  &= \frac12\, m_2[\nu-\mu].
\end{align*}
\end{theorem} 
 In order to make Theorem~\ref{thm:main}  applicable, one must understand how the hockey-stick domination arises in concrete settings. This is the second central component of the paper. Starting from elementary dominated pairs on finite chains, we show how hockey-stick domination can be transported through direct products, disjoint unions, ordinal sums, and compatible ideal restrictions. These operations generate explicit dominated pairs on Boolean posets, rectangular lattices,   certain families of rooted trees and Young diagrams. The full collection of constructions is developed in Section~\ref{eq:exhsdom}. A key feature of the developed theory is that every object introduced throughout the paper is determined entirely by the order structure of the underlying poset and is consequently preserved under poset isomorphisms.
\\

\noindent
The rest of the paper is organized as follows. Section~\ref{sec:1} recalls the basic material on incidence algebras and M\"obius inversion needed throughout the paper. Section~\ref{sec:calcoverlattices} introduces the poset integral and derivative, develops the associated power functions and moment functionals, and discusses the resulting notions of M\"obius linearity and convexity. Section~\ref{sec:3.5} is devoted to hockey-stick functions and hockey-stick domination, including their basic properties, examples, and their relation to positivity tests for signed measures. Section~\ref{eq:exhsdom} develops several mechanisms for constructing hockey-stick domination, including chains, direct products, disjoint unions, ordinal sums, ideal restrictions, and examples on Boolean posets, rectangular lattices, rooted trees and Young diagrams. Section~\ref{sec:6} discusses convex domination and its relation with hockey-stick domination. Section~\ref{sec:zolotarev} introduces the poset version of the Zolotarev-type distance and compares it with total variation. Finally, Section~\ref{sec:8} is devoted to the proof of Theorem~\ref{thm:main}.

%==============================================================
\section{Preliminaries}\label{sec:1}
For the sake of self-containment and to fix notation, we briefly review the algebraic-combinatorial notions used throughout, within the framework of incidence algebras developed by Gian-Carlo Rota \cite{n11}. Let $\mathcal{P}$ be a finite partially ordered set endowed with a partial order $\leq_{\mathcal{P}}$. We denote by
\begin{equation*}
\mathcal{P}^{(2)} := \{\, (x,y) \in \mathcal{P}^2 \;;\; x \leq_{\mathcal{P}} y \,\}
\end{equation*}
the set of comparable pairs. For a given pair of functions $F,G : \mathcal{P}^{(2)} \to \mathbb{C}$, we define their convolution $F * G$ as the function from $\mathcal{P}^{(2)}$ to $\mathbb{C}$ defined by
\begin{equation*}
(F * G)(x,y)
:= \sum_{\substack{z \in \mathcal{P} \\ x \leq_{\mathcal P} z \leq_{\mathcal P} y}}
F(x,z)\, G(z,y).
\end{equation*}
We extend the definition to the case where $F$ is replaced by a function $f : \mathcal{P} \to \mathbb{C}$ through 
\begin{equation*}
(f * G)(x)
:= \sum_{\substack{y \in \mathcal{P} \\ y \leq_{\mathcal P} x}}
f(y)\, G(y,x).
\end{equation*}
These convolution operations possess the natural algebraic properties expected from an integral-type construction: they are associative and distributive with respect to linear combinations of functions on $\mathcal{P}^{(2)}$ and on $\mathcal{P}$. The function $\delta : \mathcal{P}^{(2)} \to \mathbb{C}$ defined by
\begin{equation*}
\delta(x,y) =
\begin{cases}
1 & \text{if } x = y, \\
0 & \text{if } x <_{\mathcal P} y,
\end{cases}
\end{equation*}
acts as the unit element for convolution, in the sense that
\[
F * \delta = \delta * F = F.
\]
The zeta function of $\mathcal{P}$ is the function 
$\zeta : \mathcal{P}^{(2)} \to \mathbb{C}$ defined by
\begin{equation*}
\zeta(x,y) = 1,
\end{equation*}
for all $(x,y)$ in $\mathcal{P}^{(2)}$. The inverse of $\zeta$ under convolution is called the M\"obius function of $\mathcal{P}$, is defined over $\mathcal{P}^{(2)}$ and is denoted by $\mu_{\mathcal P}$. Observe that $\mu_{\mathcal P}$ is uniquely determined by the relation 
$\mu_{\mathcal P} * \zeta = \delta$, which amounts to the following system 
of equations for the values of $\mu_{\mathcal P}$:
\begin{equation*}
\sum_{\substack{y \in \mathcal{P} \\ x \,\leq_{\mathcal P} \, y \, \leq_{\mathcal P} \, z}} \mu_{\mathcal P}(x,y)
=
\begin{cases}
1 & \text{if } x = z,\\
0 & \text{if } x <_{\mathcal P} z.
\end{cases}
\end{equation*}
Equivalently, $\mu_{\mathcal P}$ is determined by the relation 
$\zeta * \mu_{\mathcal P} = \delta$, which yields the system
\begin{equation} \label{gd}
\sum_{\substack{y \in \mathcal{P} \\ x \, \leq_{\mathcal P} \, y \, \leq_{\mathcal P} \, z}} \mu_{\mathcal P}(y,z)
=
\begin{cases}
1 & \text{if } x = z,\\
0 & \text{if } x <_{\mathcal P} z.
\end{cases}
\end{equation}
Building upon these concepts, we have the celebrated M\"obius inversion formula, which is presented next.
\begin{proposition} \label{thmmobiusinversion}
For two functions $f,g : \mathcal{P} \to \mathbb{C}$, the statement that
\begin{equation*}
f(x) = \sum_{\substack{y \in \mathcal{P} \\ y \leq_{\mathcal P} x }} g(y)
\end{equation*}
for all $x$ in $\mathcal{P}$ is equivalent to
\begin{equation*}
g(x) = \sum_{\substack{y \in \mathcal{P} \\ y \leq_{\mathcal P} x}} f(y) \mu_{\mathcal P}(y,x).
\end{equation*}
\end{proposition}

%==============================================================

\section{M\"obius operators on finite posets”}\label{sec:calcoverlattices}

\noindent In this section, we introduce the notions of derivative and integral over finite posets and their basic properties. Throughout, $\mathcal{P}$ denotes a fixed finite poset, and $\mathfrak{F}$ denotes the space of real-valued functions defined on $\mathcal{P}$. We write $\mathbf{1}$ for the function identically equal to one. We denote by $\mathcal{M}(\mathcal{P})$ the set of signed measures on $\mathcal{P}$ and by $\mathcal{M}_1(\mathcal{P})$ the subset of probability measures. For $f$ in $\mathfrak{F}$ and $\rho$ in $\mathcal{M}(\mathcal{P})$, we use the dual pairing notation
\begin{align*}
\langle \rho, f \rangle
:=
\int_{\mathcal{P}} f(x)\,\rho(dx).
\end{align*}
\noindent We now define the poset integral and derivative.

\begin{definition} \label{d1}
Define the operators $I_{\mathcal P},D_{\mathcal P} : \mathfrak{F} \rightarrow \mathfrak{F}$ as follows:
\begin{align*}
I_{\mathcal{P}}[f](x)
&=
f * \zeta
=
\sum_{y \in \mathcal{P}} f(y) \mathbbm{1}_{\{y \leq x\}},
\end{align*}
and
\begin{align*}
D_{\mathcal{P}}[f](x)
&=
f * \mu_{\mathcal P}
=
\sum_{y \in \mathcal{P}}  f(y) \mu_{\mathcal P}(y, x) \mathbbm{1}_{\{y \leq x\}}.
\end{align*}
\end{definition} 

\noindent When the poset under consideration is clear, we omit its dependence in the notation. By the M\"obius inversion formula, the operators $I_{\mathcal P}$ and $D_{\mathcal P}$ are inverse to each other:
\[
D_{\mathcal P}\big[I_{\mathcal P}[f]\big]
=
f
=
I_{\mathcal P}\big[D_{\mathcal P}[f]\big].
\]
Thus, $I_{\mathcal P}$ may be viewed as an integral-type operator accumulating the values of $f$ over lower intervals, while $D_{\mathcal P}$ plays the role of a derivative obtained through M\"obius inversion. In this sense, the pair $(I_{\mathcal P},D_{\mathcal P})$ satisfies a discrete analog of the fundamental theorem of calculus.

When the poset is a chain, the present construction reduces to the classical calculus of finite differences: convolution with the zeta function corresponds to cumulative summation, while convolution with the M\"obius function coincides with the usual finite-difference operator (see Subsection~\ref{scr}). This correspondence, already observed in Example~2 of~\cite{n11}, places the chain case of our framework within the classical theory of finite differences.

\begin{remark}
A significant difference between chains and general posets is that the operator $D_{\mathcal P}$ is typically nonlocal. Indeed, while finite differences on a chain depend only on neighboring values, the quantity $D_{\mathcal P}[f](x)$ depends on the values of $f$ throughout the lower interval $(-\infty,x]$, weighted by the M\"obius function. Its value is therefore determined by the structure of an entire interval rather than only by immediate cover relations.

At a conceptual level, this construction is analogous to other generalized or nonlocal calculi in which differentiation is recovered from an inverse-kernel relation. In the present setting, the zeta function plays the role of the integral kernel, while the M\"obius function acts as its inverse kernel through incidence convolution. We emphasize, however, that this analogy is only heuristic and that no fractional-calculus structure is developed in this work.
\end{remark}
\noindent The operators $I$ and $D$ are compatible with poset isomorphisms. In the sequel, we fix a poset $\mathcal{P}$ together with a poset isomorphism $T:\mathcal{P}\to\mathcal{Q}$. It readily follows, by a change of variables and the fact that $T$ preserves the order, that for every function $f:\mathcal Q\to\mathbb R$,
\begin{align*}
I_{\mathcal{P}}[ f \circ T]
  &= I_{\mathcal{Q}}[f] \circ T,
&
D_{\mathcal{P}}[f \circ T]
  &= D_{\mathcal{Q}}[f] \circ T.
\end{align*}

\subsection{Power functions, moments, and M\"obius convexity}

With these objects in mind, we define power functions of index $d$ in $\mathbb{N}$ by
\begin{align*}
\phi_{d}^{\mathcal{P}} := d! \, I_{\mathcal P}^{d}[\mathbf{1}],
\end{align*}
where $I_{\mathcal P}^{d}$ denotes the $d$-th iteration of the operator $I_{\mathcal P}$. As before, these functions are invariant under poset isomorphisms, in the sense that
\begin{align*}
\phi_{d}^{\mathcal{P}}=\phi_{d}^{\mathcal{Q}} \circ T.
\end{align*}
The moment of order $d$ of a given element $\rho$ in $\mathcal{M}(\mathcal{P})$ is then defined as
\begin{align*}
m_{d}[\rho] := \langle \rho, \phi_{d}^{\mathcal{P}} \rangle.
\end{align*}
The derivative operator naturally induces notions of linearity and convexity.

\begin{definition}
A function $f$ in $\mathfrak{F}$ is called M\"obius linear if there exists a real constant $c$ such that
\begin{align*}
D[f](x)=c
\end{align*}
for all $x$ in $\mathcal{P}$. We say that $f$ is M\"obius convex if
\begin{align*}
D[f](x)\leq D[f](y)
\end{align*}
whenever $x\leq y$ in $\mathcal{P}$.
\end{definition}

\noindent
Some of the fundamental properties of M\"obius linear and M\"obius convex functions are presented next. In the sequel, we use the notation
\begin{align*}
(-\infty,x]
  &:=\{y\in\mathcal{P}\ ;\ y\leq x\},
&
[x,\infty)
  &:=\{y\in\mathcal{P}\ ;\ x\leq y\}.
\end{align*}
We also write
\begin{align*}
\min(\mathcal P)
:=
\{w\in\mathcal P\ ;\ \text{there is no } y\in\mathcal P \text{ such that } y< w\}
\end{align*}
for the set of minimal elements of $\mathcal P$. One can easily check that
\begin{align*}
(\mathbf{1} * \zeta)(x)=|(-\infty,x]|.
\end{align*}

\begin{lemma} \label{l1}
For a given function $f$ in $\mathfrak{F}$, the following properties hold.
\begin{enumerate}
\item[$i)$] $f$ is M\"obius linear if and only if
\begin{align*}
f(x)=c\,|(-\infty,x]|
\end{align*}
for all $x$ in $\mathcal P$, with $c$ a real constant.

\item[$ii)$] If $D[f]\geq 0$, then $f$ is increasing and nonnegative.

\item[$iii)$] If $D^2[f]\geq 0$, then $f$ is M\"obius convex, increasing and nonnegative.

\item[$iv)$] If $f$ is M\"obius convex and $D[f](x)\geq 0$ for all $x$ in $\min(\mathcal P)$, then $f$ is increasing and nonnegative.

\item[$v)$] If $f$ is M\"obius convex and $D^2[f]\leq 0$, then $D[f]$ is constant and nonpositive on each connected component of the Hasse diagram of $\mathcal{P}$.
\end{enumerate}
\end{lemma}

\begin{proof}
If $f$ is M\"obius linear, then $D[f]=c$ for some constant $c \in \mathbb{R}$. Since $f=I[D[f]]$, it follows that
\begin{align*}
f(x)=I[c](x)=\sum_{z\le x} c
= c\,|(-\infty,x]|.
\end{align*}
The converse is immediate from $D[I[c]]=c$.\\

\noindent For the second statement, assume $D[f]\ge0$ and let $x<y$. Using $f=I[D[f]]$, we obtain
\begin{align*}
f(y)-f(x)
&= \sum_{z\le y}D[f](z)-\sum_{z\le x}D[f](z)  \\
&= \sum_{\substack{z\le y\\ z\nleq x}} D[f](z)\ge0.
\end{align*}
Thus $f$ is increasing. Finally, note that
\begin{align*}
f=I[D[f]]\geq 0.
\end{align*}

\noindent The third statement follows by applying $ii)$ with $F=D[f]$, which gives that $D[f]$ is increasing and nonnegative. Therefore $f$ is M\"obius convex, and another application of $ii)$ gives that $f$ is increasing and nonnegative.\\

\noindent For the fourth statement, assume that $f$ is M\"obius convex and that $D[f](x)\ge0$ for every $x$ in $\min(\mathcal P)$. Since $D[f]$ is increasing, every $z$ in $\mathcal P$ lies above at least one minimal element, and therefore $D[f](z)\ge0$. The result then follows from $ii)$.\\

\noindent For the last part, observe that by hypothesis $D^2[-f]\ge0$. By $iii)$, this implies that $-f$ is M\"obius convex, increasing and nonnegative. In particular, $D[-f]$ is increasing. Therefore $D[f]$ is decreasing. On the other hand, since $f$ is M\"obius convex, $D[f]$ is increasing. Hence $D[f]$ is both increasing and decreasing and therefore it is constant along comparable pairs. It follows that $D[f]$ is constant on each connected component of the Hasse diagram of $\mathcal{P}$. Since $-f$ is nonnegative, this constant is nonpositive on each connected component.
\end{proof}

\noindent It is useful to contrast M\"obius convexity with convexity notions already present in the literature, especially those formulated on lattices. Recall that a partially ordered set is called a lattice if every pair of elements admits a least upper bound $x\vee y$ and a greatest lower bound $x\wedge y$; see \cite{stanley2011enumerative,trotter1992combinatorics}. While our setting does not assume such structure, whenever $\mathcal{P}$ is a lattice one may compare M\"obius convexity with supermodularity \cite{topkis1998,murota2003,stanley2011enumerative}. Recall that $f:\mathcal{P}\to\mathbb{R}$ is supermodular if
\begin{align*}
f(x\vee y)+f(x\wedge y)\ge f(x)+f(y),
\end{align*}
for $x,y$ in $\mathcal{P}$. Although both supermodularity and M\"obius convexity are naturally defined on lattices, the relationship between them is, in general, nontrivial. For instance, on a chain, M\"obius convexity coincides with the notion of convexity in finite difference calculus, whereas supermodularity becomes vacuous.

\section{Hockey-Stick functions}\label{sec:3.5}

\noindent We now introduce the family of test functions that will induce the domination relation used in the comparison of probability measures on $\mathcal P$. For $x,z$ in $\mathcal P$, define
\begin{equation}\label{eq:Phixhockeydef}
\Phi_x^{\mathcal P}(z)
:=
\sum_{y\in\mathcal P}
\mathbbm{1}_{\{\,x\le y\le z\,\}}.
\end{equation}
Thus, $\Phi_x^{\mathcal P}(z)$ counts the number of elements $y$ such that $x\le y\le z$. This definition is the poset counterpart of the usual hockey-stick function on the real line, since
\begin{align*}
(z-x)_+
=
\int_{\mathbb R}\mathbbm{1}_{\{\,x\le y\le z\,\}}\,dy.
\end{align*}
In convex analysis, optimization, machine learning, and stochastic dominance theory, these real-line hockey-stick functions are also known as hinge functions \cite{boyd2004convex,niculescu2006convex,shaked2007stochastic}. We keep the term hockey-stick function in order to emphasize the connection with the domination relation developed below.\\

\noindent The relevance of hockey-stick functions comes from their role as elementary building blocks for convex functions. In particular, under the usual representation of one-dimensional convex functions, and under the corresponding finiteness assumptions, one may write
\begin{equation} \label{cr}
g(x)
=
\alpha+\beta x+\int_{\mathbb R}(x-\gamma)_+\,\mu(d\gamma),
\end{equation}
for suitable constants $\alpha,\beta$ in $\mathbb R$ and a nonnegative measure $\mu$. This representation is closely related to the classical characterization of convex order: for integrable random variables $X$ and $Y$, the property
\begin{equation}\label{ce1}
    \mathbb{E}[g(X)] \le \mathbb{E}[g(Y)],
\end{equation}
for all convex functions $g$, is equivalent to
\begin{equation} \label{ce2}
    \mathbb{E}[(X-\gamma)_+] \le \mathbb{E}[(Y-\gamma)_+] 
    \quad \text{for all } \gamma\in\mathbb R,
    \quad \text{and} \quad
    \mathbb{E}[X] = \mathbb{E}[Y].
\end{equation}
This observation motivates the use of hockey-stick functions as poset analogs of the real-line hockey-stick functions, and as a first family of tests from which domination relations can be built.

\subsection{Basic properties}
In this subsection we present some of the basic properties of hockey-stick functions, defined through \eqref{eq:Phixhockeydef}. First we observe that, by a direct computation,
\begin{align*}
(\zeta * \zeta)(x,z)=|[x,z]|=\Phi_x^{\mathcal{P}}(z).
\end{align*}
Moreover, as in the case of the power functions, the hockey-stick functions are invariant under isomorphisms, in the sense that
\begin{align*}
\Phi_x^{\mathcal{P}}
  &= \Phi_{T(x)}^{\mathcal{Q}} \circ T.
\end{align*}
The next result collects some elementary properties of hockey-stick functions. In what follows, for $f \in \mathfrak{F}$ we denote
\begin{align*}
D^{2+}[f] := \{ x \in \mathcal{P} : D^2[f](x) > 0 \}.
\end{align*}

\begin{proposition} \label{p0}
\begin{enumerate}
\item[$i)$] For every $x$ in $\mathcal{P}$, the function $\Phi^{\mathcal{P}}_x$ is increasing and M\"obius convex. In particular,
\begin{align*}
\Phi_x^{\mathcal P}(z)=I^2[\delta_x](z),
\end{align*}
where $\delta_x(z):=\mathbbm 1_{\{z=x\}}$. Moreover, when $\mathcal{P}$ is a lattice, $\Phi_x^{\mathcal P}$ is supermodular.

\item[$ii)$] For any function $f$ in $\mathfrak{F}$, it follows that
\begin{align*}
f= \sum_ {x \in \mathcal{P}}D^2[f](x) \,\Phi_{x}^{\mathcal P}.
\end{align*}
In particular, every function $f \in \mathfrak{F}$ can be written as the difference of two nonnegative increasing M\"obius convex functions.
\end{enumerate}
\end{proposition}

\begin{proof}
Given $x$ in $\mathcal{P}$, by definition, it is straightforward to see that
\begin{align*}
\Phi^{\mathcal{P}}_x(z)=I^2[\delta_x](z),
\end{align*}
where $\delta_x(z):=\mathbbm 1_{\{z=x\}}$.
Thus,
\begin{align*}
D^2[\Phi^{\mathcal{P}}_x](z)=\delta_x(z)\geq 0.
\end{align*}
By Lemma~\ref{l1}~$iii)$, it is M\"obius convex and increasing.

\noindent On the other hand, when $\mathcal{P}$ is a lattice, given $x,z_1,z_2$ in $\mathcal{P}$, we have
\begin{align*}
[x,z_1] \cap [x,z_2]=[x, z_1 \wedge z_2],
\end{align*}
and
\begin{align*}
[x,z_1] \cup [x,z_2]\subseteq [x, z_1 \vee z_2].
\end{align*}
By the inclusion-exclusion principle,
\begin{align*}
|[x,z_1]|+|[x,z_2]|
&=
|[x,z_1]\cup [x,z_2]|+|[x,z_1]\cap [x,z_2]|  \\
&\le
|[x,z_1\vee z_2]|+|[x,z_1\wedge z_2]|.
\end{align*}
Therefore,
\begin{align*}
\Phi_x^{\mathcal P}(z_1)+\Phi_x^{\mathcal P}(z_2)
\le
\Phi_x^{\mathcal P}(z_1\vee z_2)
+
\Phi_x^{\mathcal P}(z_1\wedge z_2),
\end{align*}
which proves supermodularity.\\

\noindent For the second statement, first observe that 
\begin{align*}
\Phi^{\mathcal{P}}_x(z)=I[g_x](z),
\end{align*}
where $g_x(y):=\mathbbm{1}_{\{\, x \leq y\}}$. Now, if we define
\begin{align*}
F :=  \sum_ {x \in \mathcal{P}}D^2[f](x) \,\Phi_{x}^{\mathcal P},
\end{align*}
then
\begin{align*}
D[F](z)
&= \sum_{x \in \mathcal{P}}  D^2[f](x) \, D[\Phi_x^{\mathcal P}](z)
= \sum_{x \in \mathcal{P}}  D^2[f](x) \, \mathbbm{1}_{\{x \le z\}}  
= I[D^2[f]](z).
\end{align*}
Consequently, $D[F] = D[f]$ and hence since $D$ is invertible by M\"obius inversion, it follows that $F = f$. Thus, we have
\begin{equation}\label{ec1}
f(z)
=
\sum_{x \in D^{2+}[f]} D^2[f](x) \, \Phi_x^{\mathcal P}(z)
-
\sum_{x \in D^{2+}[-f]} D^2[-f](x) \, \Phi_x^{\mathcal P}(z).
\end{equation}
Since each $\Phi_x^{\mathcal P}$ is nonnegative, increasing and M\"obius convex, the proof is done.
\end{proof}

\noindent
The second statement of the previous proposition has, in some sense, the same flavor as \eqref{cr}, and suggests that, at least from a theoretical perspective, many properties of functions $f\in \mathfrak{F}$ rely on an adequate understanding of hockey-stick functions and the second M\"obius derivative. In fact, it provides an alternative proof of Lemma~\ref{l1}~$iii)$. Consequently, the preceding results naturally position hockey-stick functions as a fundamental and meaningful class of functions, serving as basic building blocks when combined with the M\"obius derivative operator.

%========================================================================
\subsection{Hockey-Stick domination} \label{subhs}
We now use these functions to define a comparison relation.

\begin{definition}
We say that $\mu$ in $\mathcal{M}(\mathcal{P})$ is dominated by $\nu$ in $\mathcal{M}(\mathcal{P})$ in the hockey-stick order, written $\mu \preceq_{HS} \nu$, if for every $x$ in $\mathcal{P}$ it holds that
\begin{align*}
\langle \nu-\mu,\Phi_x^{\mathcal{P}}\rangle \ge 0.
\end{align*}
\end{definition}

\begin{remark}
Besides the plain definition, the ordering $\preceq_{HS}$ admits an interpretation in terms of the signed measure $\nu-\mu$. The tail of this measure will be denoted by
\begin{align*}
S(z):=\sum_{y\ge z} (\nu-\mu)(y),
\end{align*}
which codifies the mass discrepancy between the underlying measures, accumulated over intervals of the form $[z,\infty)$. In terms of this notation, the condition of hockey-stick domination reads
\begin{align*}
\sum_{z\ge x} S(z)\ge 0,
\end{align*}
for all $x$ in $\mathcal{P}$. This inequality says that negative contributions in upper intervals must be compensated by positive mass in the corresponding upper tails.
\end{remark}

\noindent The condition of hockey-stick domination is also invariant under isomorphism, in the sense that $\mu \preceq_{HS} \nu$ if and only if $T_{\#}\mu \preceq_{HS} T_{\#}\nu$.\\

\noindent We shall also use a localized version of the same notation. If $A\subseteq\mathcal P$, we write $\mu\preceq_{HS}\nu$ on $A$ when
\begin{align*}
\langle \nu-\mu,\Phi_x^{\mathcal P}\rangle \ge 0
\end{align*}
for every $x$ in $A$.\\

\noindent Beyond its role in bounding distances between probability measures, hockey-stick domination can also be used to extract additional information. The next result shows that the positivity of a function under the action of $\nu - \mu$ can be deduced from hockey-stick domination on the part of the poset determined by the sign of the second M\"obius derivative of the function under consideration.

\begin{corollary} \label{cc1}
Let $\mu, \nu \in \mathcal{M}(\mathcal{P})$ and $f \in \mathfrak{F}$. If the following conditions hold:
\begin{align*}
\mu \preceq_{HS} \nu 
\quad \text{on } D^{2+}[f],
\end{align*}
and
\begin{align*}
\nu \preceq_{HS} \mu 
\quad \text{on } D^{2+}[-f],
\end{align*}
then it follows that
\begin{align*}
\langle \nu-\mu,f\rangle \ge 0.
\end{align*}
\end{corollary}

\noindent The proof is a direct consequence of Proposition~\ref{p0}~$ii)$.%==============================================================

\section{Constructing Hockey-Stick Domination}\label{eq:exhsdom}

\noindent
In this section, we develop methods for constructing hockey-stick domination. Our goal is to produce explicit dominated pairs on finite posets, beginning with elementary examples and then propagating the domination relation through standard poset operations. We start with finite chains, where the construction is inherited from classical convex domination, and then show that hockey-stick domination is preserved under direct products, disjoint unions, ordinal sums, and suitable restrictions to order ideals.

\subsection{Chains}\label{eq:secchains}
Our first basic source of hockey-stick domination comes from classical convex domination on the real line. To arrange the setting adequately, let $X=\{X_k\ ;\ k\geq 1\}$ be a martingale with respect to a filtration $\mathcal{F}$, taking values in $C_n$. A fundamental example is given by a simple random walk stopped at a suitable stopping time. Let $p<q$ be integers, and let $\mu$ and $\nu$ be the probability distributions of $X_p$ and $X_q$, respectively. By the conditional Jensen inequality, for every $x$ in $\mathbb R$,
\begin{align*}
\E[(X_q-x)_{+}]
&=
\E\big[\E[(X_q-x)_{+}\mid \mathcal{F}_p]\big]  \\
&\geq
\E\big[(\E[X_q\mid \mathcal{F}_p]-x)_{+}\big]
=
\E[(X_p-x)_{+}].
\end{align*}
Using the fact that the restrictions of the functions $y\mapsto (y-x)_{+}$ to $C_n$ are related to the hockey-stick functions on $C_n$, we conclude that
\[
\mu\preceq_{HS}\nu.
\]
Since $C_n$ is isomorphic to any chain of size $n$, this yields a canonical construction of hockey-stick domination on finite chains.\\

\noindent
The previous example, while illustrative, relies on the particularly simple structure of chains. Nevertheless, it provides a useful building block for generating domination relations on richer posets through suitable binary operations. This perspective is developed in the following sections.

\subsection{Direct products} \label{sub5.2}
The first mechanism for constructing new hockey-stick dominations from old ones is through direct products. Let $(\mathcal{P},\leq_{\mathcal{P}})$ and $(\mathcal{Q},\leq_{\mathcal{Q}})$ be two posets. The direct product of $\mathcal{P}$ and $\mathcal{Q}$ is the poset
\[
\mathcal{P} \times \mathcal{Q}
:=
\{(x,y)\mid x\in \mathcal{P},\; y\in \mathcal{Q}\},
\]
endowed with the partial order defined by
\[
(x_1,y_1)\leq (x_2,y_2)
\quad\text{if and only if}\quad
x_1\leq_{\mathcal{P}} x_2
\ \text{and}\
y_1\leq_{\mathcal{Q}} y_2.
\]
The product of higher order is defined inductively. The following result establishes how to obtain hockey-stick domination in product posets from individual dominations in each component.

\begin{proposition}\label{p1}
For $i=1,\dots,n$, let $ (\mathcal{P}_i,\leq_{i})$ be finite posets and let
$\mu_i,\nu_i$ be elements of $\mathcal M_1(\mathcal P_i)$ satisfying $\mu_i\preceq_{HS}\nu_i$.
On the product poset $\mathcal P:=\prod_{i=1}^n\mathcal P_i$ endowed with its product order $\leq_{\mathcal{P}}$, we consider the measures
\[
\mu:=\otimes_{i=1}^n \mu_i,
\qquad
\nu:=\otimes_{i=1}^n \nu_i.
\]
Then $\mu\preceq_{HS}\nu$ on $\mathcal P$.
\end{proposition}

\begin{proof}
For notational convenience, we will write $\mu[y]$ and $\nu[y]$ to denote the evaluation of $\mu$ and $\nu$, respectively, over the singleton $\{y\}$. Fix an element $x=(x_1,\dots,x_n)$ in $\prod_{i=1}^n\mathcal P_i$. Our goal is to prove that 
\[
\langle \nu-\mu,\Phi_x^{\mathcal P}\rangle
\ge 0.
\]
Define, for each $k\in \{1,\dots,n\}$, the measures
\[
\eta^{(k)}:=\big(\otimes_{i=1}^k \nu_i\big)\otimes\big(\otimes_{i=k+1}^n \mu_i\big),
\]
so that $\eta^{(0)}=\mu$ and $\eta^{(n)}=\nu$. Then
\[
\nu-\mu=\sum_{k=1}^n\bigl(\eta^{(k)}-\eta^{(k-1)}\bigr),
\]
and therefore
\[
\langle \nu-\mu,\Phi_x^{\mathcal P}\rangle
=
\sum_{k=1}^n \langle \eta^{(k)}-\eta^{(k-1)},\Phi_x^{\mathcal P}\rangle.
\]
Fix $k$ in $\{1,\dots,n\}$, since $\eta^{(k)}-\eta^{(k-1)}$ only changes the $k$-th factor, we have that 
\[
\eta^{(k)}[z]-\eta^{(k-1)}[z]
=
\big(\nu_k[z_k]-\mu_k[z_k]\big)\,
\prod_{i<k}\nu_i[z_i]\,
\prod_{i>k}\mu_i[z_i].
\]
By elementary computations, we can show the factorization $\Phi_x^{\mathcal P}(z)=\prod_{i=1}^n \Phi_{x_i}^{\mathcal P_i}(z_i)$, which in combination with the previous observations yields
\begin{multline*}
\langle \eta^{(k)}-\eta^{(k-1)},\Phi_x^{\mathcal P}\rangle\\
\begin{aligned}
&=
\Bigg(\prod_{i<k}\sum_{z_i\in\mathcal P_i}\Phi_{x_i}^{\mathcal P_i}(z_i)\,\nu_i[z_i]\Bigg)
\Bigg(\prod_{i>k}\sum_{z_i\in\mathcal P_i}\Phi_{x_i}^{\mathcal P_i}(z_i)\,\mu_i[z_i]\Bigg)
\Bigg(\sum_{z_k\in\mathcal P_k}\Phi_{x_k}^{\mathcal P_k}(z_k)\,\big(\nu_k[z_k]-\mu_k[z_k]\big)\Bigg).
\end{aligned}
\end{multline*}
The first two factors in the product at the right are nonnegative, since $\Phi_{x_i}^{\mathcal P_i}\ge 0$ and $\mu_i,\nu_i$ are probability measures. Moreover, by the hypothesis $\mu_k\preceq_{HS}\nu_k$ on $\mathcal P_k$, we have for every $x_k$ in $\mathcal P_k$,
\[
\sum_{z_k\in\mathcal P_k}\Phi_{x_k}^{\mathcal P_k}(z_k)\,\big(\nu_k[z_k]-\mu_k[z_k]\big)\ge 0.
\]
This proves the result.
\end{proof}

\noindent An alternative mechanism for producing hockey-stick domination on product spaces is obtained by arranging products along a chain. This construction relies on the structure of the associated Hasse diagram: the diagram of $\mathcal{P}\times C_n$ consists of $n$ layered copies of the Hasse diagram of $\mathcal{P}$, with additional cover edges connecting $(x,i)$ to $(x,i+1)$ for $x$ in $\mathcal{P}$ and $i=1,\dots,n-1$. In this way, the product poset appears as a vertical stacking of $\mathcal{P}$, and domination on $\mathcal{P}\times C_n$ is obtained by propagating the inequalities along the chain direction. The following proposition formalizes this construction.

\begin{proposition}\label{p1c}
Let $(\mathcal{P},\leq_{\mathcal{P}})$ be a finite poset and let 
$\mu_i,\nu_i$ be elements in $\mathcal M_1(\mathcal P)$ satisfying 
$\mu_i \preceq_{HS} \nu_i$ on $\mathcal P$ for each $i=1,\dots,n$.
Let $\alpha_1,\dots,\alpha_n\ge0$ satisfy $\sum_{i=1}^n\alpha_i=1$.
Define probability measures $\mu,\nu $ on $ \mathcal{P} \times C_n $ by
\[
\mu[(x,i)] := \alpha_i\mu_i[x]
\quad\text{and}\quad
\nu[(x,i)] := \alpha_i\nu_i[x].
\]
Then $\mu \preceq_{HS} \nu$ on $\mathcal{P}\times C_n$.
\end{proposition}

\begin{proof}
We must show that for every $(x,i) $ in $ \mathcal{P}\times C_n$,
\[
\langle \nu-\mu, \Phi_{(x,i)}^{\mathcal{P}\times C_n} \rangle
= \sum_{(z,j) \in \mathcal{P}\times C_n } \Phi_{(x,i)}^{\mathcal{P}\times C_n} (z,j)\bigl(\nu[(z,j)]-\mu[(z,j)]\bigr) \ge 0.
\]
By the definition of $\mu$ and $\nu$, this can be written as
\[
\sum_{(z,j) \in \mathcal{P}\times C_n } 
\Phi_{(x,i)}^{\mathcal{P}\times C_n}(z,j)\,
\bigl(\nu[(z,j)]-\mu[(z,j)]\bigr)
=
\sum_{j=1}^n \alpha_j \sum_{z \in \mathcal{P}} 
\Phi_{(x,i)}^{\mathcal{P}\times C_n}(z,j)\,
\bigl(\nu_j[z]-\mu_j[z]\bigr).
\]
Now, by definition,
\[
\begin{aligned}
\Phi_{(x,i)}^{\mathcal{P}\times C_n}(z,j)
&=
\sum_{(y,k)\in \mathcal{P}\times C_n}
\mathbbm{1}_{\{x\le_{\mathcal{P}} y \le_{\mathcal{P}} z\}}
\mathbbm{1}_{\{i\le_{C_n} k \le_{C_n} j\}} \\[0.3em]
&=
\Phi_x^{\mathcal{P}}(z)\,\Phi_i^{C_n}(j).
\end{aligned}
\]
As a consequence, for every $(x,i) $ in $ \mathcal{P}\times C_n$,
\[
\begin{aligned}
\sum_{j=1}^n \alpha_j \sum_{z \in \mathcal{P}} 
\Phi_{(x,i)}^{\mathcal{P}\times C_n}(z,j)\,
\bigl(\nu_j[z]-\mu_j[z]\bigr)
&=
\sum_{j=1}^n \alpha_j \sum_{z \in \mathcal{P}} 
\Phi_i^{C_n}(j)\Phi_x^{\mathcal{P}}(z)\,
\bigl(\nu_j[z]-\mu_j[z]\bigr) \\[0.3em]
&=
\sum_{j=1}^n \alpha_j\Phi_i^{C_n}(j)  
\sum_{z \in \mathcal{P}} \Phi_x^{\mathcal{P}}(z)\,
\bigl(\nu_j[z]-\mu_j[z]\bigr).
\end{aligned}
\]
By hypothesis, $\mu_j \preceq_{HS} \nu_j$ on $\mathcal P$ for each $j=1,\dots,n$; that is, for every $x$ in $\mathcal{P}$,
\[
\sum_{z\in\mathcal{P}} \Phi_x^{\mathcal{P}}(z)\,
\bigl(\nu_j[z]-\mu_j[z]\bigr) \ge 0.
\]
Together with the facts that $\alpha_j\ge0$ and $\Phi_i^{C_n}(j)\ge0$, this proves the result.
\end{proof}

\subsubsection{The Boolean poset}
The combination of chains and direct products of posets naturally leads to the formulation of posets possessing hockey-stick dominated measures. In particular, for every collection of probability measures $(\mu_i,\nu_i)_{i=1}^n$ on $C_2$, Proposition~\ref{p1} shows that if $\mu_i \preceq_{HS} \nu_i$ for $i=1,\dots,n$, then the product measures $\mu=\mu_1\otimes\cdots\otimes\mu_n$ and $\nu=\nu_1\otimes\cdots\otimes\nu_n$ satisfy $\mu \preceq_{HS} \nu$ on $C_2^n$. Since the Boolean poset of rank $n$ is the $n$-fold product $C_2^n$, this yields a canonical construction of hockey-stick dominated measures on $\mathrm{Bool}_n$. Furthermore, for any fixed $1 \le i \le n-1$ we have
\begin{align*}
\mathrm{Bool}_n \cong \mathrm{Bool}_i \times C_2^{\,n-i},
\end{align*}
and therefore Propositions~\ref{p1} and~\ref{p1c} provide different ways to construct hockey-stick domination on $\mathrm{Bool}_n$.

\subsubsection{Rectangular lattice}
\noindent In this setting, Propositions~\ref{p1} and~\ref{p1c} reduce the construction of hockey-stick domination on \(C_{n_1}\times C_{n_2}\) to the corresponding construction on chains. We therefore record explicitly the chain inequalities. Consider measures \(\mu,\nu\) in $\mathcal M(C_m)$ satisfying \(\mu\preceq_{HS}\nu\). By definition, this condition requires that for every \(i=1,\dots,m\),
\begin{align*}
\sum_{j\in C_m}\Phi_i^{C_m}(j)\,
\bigl(\nu[j]-\mu[j]\bigr)\ge0.
\end{align*}
Recall that
\begin{align*}
\Phi_i^{C_m}(j)=
\begin{cases}
0, & i>j,\\[6pt]
j-i+1, & i\le j.
\end{cases}
\end{align*}
Therefore, for each fixed \(i\),
\begin{align*}
\sum_{j\in C_m}\Phi_i^{C_m}(j)\,
\bigl(\nu[j]-\mu[j]\bigr)
&=
\sum_{j=i}^{m}\Phi_i^{C_m}(j)\,
\bigl(\nu[j]-\mu[j]\bigr)  =
\sum_{l=0}^{m-i}(l+1)\,
\bigl(\nu[i+l]-\mu[i+l]\bigr).
\end{align*}
This reformulation shows that the problem reduces to finding parameters 
\(a_j,b_j \) in $\mathbb{R}$ with $j=1,\dots,m$, such that 
\begin{equation}\label{e1}
\sum_{l=0}^{m-i}(l+1)\,
\bigl(b_{i+l}-a_{i+l}\bigr)\ge0 \quad \text{for} \quad  i=1,\dots,m.
\end{equation}
Expression \eqref{e1} forms a triangular system of linear inequalities in the variables \(d_j:=b_j-a_j\), in the sense that the constraint indexed by \(i\) involves only the variables \(d_j\) with \(j\ge i\). In particular,
\begin{align*}
d_m\ge0,\qquad
d_{m-1}+2d_m\ge0,\qquad
d_{m-2}+2d_{m-1}+3d_m\ge0,
\end{align*}
and so on. Hence the feasible region for the differences \(d_j\) may be constructed inductively starting from the top index \(j=m\) and proceeding downward. Consequently, every pair of probability measures \((\mu,\nu)\) exhibiting hockey-stick domination on \(C_m\) is described by parameters \(a_j,b_j\ge0\), for \(j=1,\dots,m\), satisfying
\begin{align*}
\sum_{j=1}^m a_j=\sum_{j=1}^m b_j=1,
\end{align*}
together with \eqref{e1}, via the identification
\begin{align*}
\mu[j]=a_j,
\qquad
\nu[j]=b_j.
\end{align*}

\subsection{Disjoint unions}
Another natural instance where hockey-stick dominations can be constructed from existing ones is through disjoint unions. 
Let $(\mathcal{P},\leq_{\mathcal{P}})$ and $(\mathcal{Q},\leq_{\mathcal{Q}})$ be two posets. The disjoint union of $\mathcal{P}$ and $\mathcal{Q}$ is the poset obtained as the disjoint union $\mathcal{P}\sqcup\mathcal{Q}$, endowed with the order 
\[
x \leq y
\quad\text{if and only if}\quad
\begin{cases}
x,y\in\mathcal{P}\ \text{and}\ x\leq_{\mathcal{P}} y,\\
x,y\in\mathcal{Q}\ \text{and}\ x\leq_{\mathcal{Q}} y.
\end{cases}
\]
In particular, elements belonging to different components are incomparable. 
Higher-order disjoint unions are defined inductively. As in the previous section, the disjoint union operation preserves the structure of hockey-stick domination, as stated in the following result.

\begin{proposition}\label{p:disjoint}
Let $(\mathcal{P}_i,\leq_i)$, $i=1,\dots,n$, be finite posets and $\mu_i\preceq_{HS}\nu_i$ on $\mathcal{P}_i$  for each $i$. 
Fix weights $\alpha_1,\dots,\alpha_n\ge0$ and define measures on the disjoint union
\[
\mathcal{P}:=\bigsqcup_{i=1}^n\mathcal{P}_i
\]
by
\[
\mu := \sum_{i=1}^n \alpha_i \mu_i,
\qquad
\nu := \sum_{i=1}^n \alpha_i \nu_i,
\]
where each $\mu_i,\nu_i$ is extended by zero outside $\mathcal{P}_i$. Then $\mu \preceq_{HS} \nu$.
\end{proposition}

\begin{proof}
Fix $x$ in $\mathcal{P}$. Then there exists a unique index $k$ such that $x$ belongs to $\mathcal{P}_k$. 
By definition of the disjoint union order, intervals do not mix components, and therefore
\[
\Phi_x^{\mathcal{P}}(z)
=
\begin{cases}
\Phi_x^{\mathcal{P}_k}(z), & z\in\mathcal{P}_k,\\
0, & z\notin\mathcal{P}_k.
\end{cases}
\]
Consequently,
\[
\langle \nu-\mu,\Phi_x^{\mathcal{P}}\rangle
=
\sum_{z\in\mathcal{P}_k}
\Phi_x^{\mathcal{P}_k}(z)\,
\bigl(\nu[z]-\mu[z]\bigr)=
\alpha_k
\sum_{z\in\mathcal{P}_k}
\Phi_x^{\mathcal{P}_k}(z)\,
\bigl(\nu_k[z]-\mu_k[z]\bigr).
\]
Since $\alpha_k\ge0$ and $\mu_k\preceq_{HS}\nu_k$ on $\mathcal{P}_k$, we have
\[
\sum_{z\in\mathcal{P}_k}
\Phi_x^{\mathcal{P}_k}(z)\,
\bigl(\nu_k[z]-\mu_k[z]\bigr)\ge0,
\]
which implies $\langle \nu-\mu,\Phi_x^{\mathcal{P}}\rangle \ge 0$, as required.
\end{proof}

\subsection{Ordinal sums}

A further mechanism through which new hockey-stick dominations may be constructed from existing ones is provided by ordinal sums. Let $(\mathcal{P},\leq_{\mathcal{P}})$ and $(\mathcal{Q},\leq_{\mathcal{Q}})$ be two posets. The ordinal sum of $\mathcal{P}$ and $\mathcal{Q}$ is the poset obtained from the disjoint union $\mathcal{P}\sqcup\mathcal{Q}$, denoted by $\mathcal{P}\oplus\mathcal{Q}$, and endowed with the partial order defined by
\begin{align*}
x \leq y\quad 
\quad\text{if and only if}\quad\quad 
\begin{cases}
x,y\in\mathcal{P}\ \text{and}\ x\leq_{\mathcal{P}} y,\\
x,y\in\mathcal{Q}\ \text{and}\ x\leq_{\mathcal{Q}} y,\\
x\in\mathcal{P},\ y\in\mathcal{Q}.
\end{cases}
\end{align*}
In particular, every element of $\mathcal{P}$ is declared to be below every element of $\mathcal{Q}$. Higher-order ordinal sums are defined inductively. As in the previous sections, this operation is compatible with hockey-stick domination.

\begin{proposition}\label{p:ordinal}
Let $(\mathcal P,\leq_{\mathcal P})$ and $(\mathcal Q,\leq_{\mathcal Q})$ be finite posets, and assume that $\mathcal Q$ has a least element $0_{\mathcal Q}$. Suppose that $\mu_{\mathcal P},\nu_{\mathcal P}\in\mathcal M_1(\mathcal P)$ and $\mu_{\mathcal Q},\nu_{\mathcal Q}\in\mathcal M_1(\mathcal Q)$ satisfy
\begin{align*}
\mu_{\mathcal P}\preceq_{HS}\nu_{\mathcal P}
\quad\text{on } \mathcal P,
\qquad
\mu_{\mathcal Q}\preceq_{HS}\nu_{\mathcal Q}
\quad\text{on } \mathcal Q.
\end{align*}
Let $\alpha,\beta\ge 0$ satisfy $\alpha+\beta=1$, and define probability measures on the ordinal sum $\mathcal P\oplus \mathcal Q$ by
\begin{align*}
\mu=\alpha\,\mu_{\mathcal P}+\beta\,\mu_{\mathcal Q},
\qquad
\nu=\alpha\,\nu_{\mathcal P}+\beta\,\nu_{\mathcal Q},
\end{align*}
where each component measure is extended by zero outside its original support. Then
\begin{align*}
\mu \preceq_{HS} \nu
\end{align*}
on $\mathcal P\oplus \mathcal Q$.
\end{proposition}

\begin{proof}
Fix $x$ in $\mathcal{P}\oplus\mathcal{Q}$. We split the proof into two parts.

\medskip

\noindent
\textit{Case I.}
First we consider the case in which $x$ belongs to $\mathcal{P}$. In this instance, intervals extending above $x$ may lie either inside $\mathcal{P}$ or inside $\mathcal{Q}$, and therefore
\begin{align*}
\Phi_x^{\mathcal{P}\oplus\mathcal{Q}}(z)
=
\begin{cases}
\Phi_x^{\mathcal{P}}(z), & z\in\mathcal{P},\\[4pt]
|[x,\infty)_{\mathcal P}| + |(-\infty,z]_{\mathcal Q}|, & z\in\mathcal{Q}.
\end{cases}
\end{align*}
Consequently,
\begin{align*}
\langle \nu-\mu,\Phi_x^{\mathcal{P}\oplus\mathcal{Q}}\rangle
&=
\sum_{z\in\mathcal{P}}
\Phi_x^{\mathcal{P}}(z)\,
\bigl(\nu[z]-\mu[z]\bigr)
+
\sum_{z\in\mathcal{Q}}
\Big(|[x,\infty)_{\mathcal P}| + |(-\infty,z]_{\mathcal Q}|\Big)\,
\bigl(\nu[z]-\mu[z]\bigr)\\
&=
\alpha
\sum_{z\in\mathcal{P}}
\Phi_x^{\mathcal{P}}(z)\,
\bigl(\nu_{\mathcal{P}}[z]-\mu_{\mathcal{P}}[z]\bigr)
+
\beta\,|[x,\infty)_{\mathcal P}|
\sum_{z\in\mathcal{Q}}
\bigl(\nu_{\mathcal{Q}}[z]-\mu_{\mathcal{Q}}[z]\bigr)\\
&\qquad
+
\beta
\sum_{z\in\mathcal{Q}}
|(-\infty,z]_{\mathcal{Q}}|\,
\bigl(\nu_{\mathcal{Q}}[z]-\mu_{\mathcal{Q}}[z]\bigr).
\end{align*}
Since $\mu_{\mathcal Q}$ and $\nu_{\mathcal Q}$ are probability measures, the middle term is equal to zero. On the other hand, for every $z$ in $\mathcal Q$,
\begin{align*}
|(-\infty,z]_{\mathcal Q}|=|[0_{\mathcal Q},z]_{\mathcal Q}|=\Phi_{0_{\mathcal Q}}^{\mathcal Q}(z),
\end{align*}
and therefore, by $\mu_{\mathcal Q}\preceq_{HS}\nu_{\mathcal Q}$,
\begin{align*}
\sum_{z\in\mathcal{Q}}
|(-\infty,z]_{\mathcal{Q}}|\,
\bigl(\nu_{\mathcal{Q}}[z]-\mu_{\mathcal{Q}}[z]\bigr)
&=
\sum_{z\in\mathcal{Q}}
\Phi_{0_{\mathcal Q}}^{\mathcal Q}(z)\,
\bigl(\nu_{\mathcal{Q}}[z]-\mu_{\mathcal{Q}}[z]\bigr)\ge 0.
\end{align*}
Finally, the first term is nonnegative because $\mu_{\mathcal{P}}\preceq_{HS}\nu_{\mathcal{P}}$. Hence
\begin{align*}
\langle \nu-\mu,\Phi_x^{\mathcal{P}\oplus\mathcal{Q}}\rangle\ge 0.
\end{align*}

\medskip

\noindent
\textit{Case II.}
Consider now the case $x$ in $\mathcal{Q}$. For this instance, intervals remain entirely inside $\mathcal{Q}$, and therefore
\begin{align*}
\Phi_x^{\mathcal{P}\oplus\mathcal{Q}}(z)
=
\begin{cases}
\Phi_x^{\mathcal{Q}}(z), & z\in\mathcal{Q},\\
0, & z\in\mathcal{P}.
\end{cases}
\end{align*}
Hence
\begin{align*}
\langle \nu-\mu,\Phi_x^{\mathcal{P}\oplus\mathcal{Q}}\rangle
&=
\beta
\sum_{z\in\mathcal{Q}}
\Phi_x^{\mathcal{Q}}(z)\,
\bigl(\nu_{\mathcal{Q}}[z]-\mu_{\mathcal{Q}}[z]\bigr)\ge 0,
\end{align*}
since $\mu_{\mathcal{Q}}\preceq_{HS}\nu_{\mathcal{Q}}$. This completes the proof.
\end{proof}

\subsubsection{Trees}
The combination of chains with the operations of disjoint union and ordinal-type attachments yields posets with branching structures while preserving hockey-stick dominated probability measures. Starting from measures on finite chains, disjoint unions produce forest-type posets, whereas attaching a root below a forest generates rooted tree structures. Structurally, every finite rooted tree poset $T$ admits a recursive decomposition
\begin{align*}
T
\;\cong\;
\mathbf 1
\;\oplus\;
\Big(\bigsqcup_{i=1}^k T_i\Big),
\end{align*}
where $\mathbf 1$ denotes the root and $T_1,\dots,T_k$ are the rooted subtrees above it. Although the upper forest $\bigsqcup_{i=1}^k T_i$ need not have a least element, the proof of Proposition~\ref{p:ordinal} applies branch by branch, together with the disjoint-union construction. Thus, if dominated pairs are given on the subtrees and the corresponding branch weights are chosen consistently for the two measures, then the resulting measures on $T$ are again hockey-stick dominated. Consequently, trees obtained through iterative disjoint unions and root attachments inherit natural families of hockey-stick dominated probability measures assembled from their chain components.

\subsection{Ideal restrictions} \label{sub5.5}
Another mechanism consists in restricting a poset to a lower set. Let $(\mathcal{P},\leq_{\mathcal{P}})$ be a finite poset. A subset $I\subseteq \mathcal{P}$ is called an order ideal if
\begin{align*}
z\in I,\ \ y\leq_{\mathcal P} z
\end{align*}
imply that  $y$ belongs to $ I$. We write $\mathcal P_{\mid I}$ for the induced subposet on $I$. Given a measure $\mu$ in $\mathcal M(\mathcal P)$, we consider its restriction to $I$, defined by
\begin{align*}
\mu_{\mid I}[z]:=\mu[z]\mathbbm 1_{\{z\in I\}},
\end{align*}
for $z$ in $\mathcal P$. Whenever $\mu$ is a probability measure and $\mu[I]>0$, we also consider the conditional restriction
\begin{align*}
\mu^{I}[z]
:=
\frac{\mu[z]}{\mu[I]}\mathbbm 1_{\{z\in I\}},
\end{align*}
for $z$ in $\mathcal P$, which defines an element of $\mathcal M_1(I)$, identified with a probability measure on $\mathcal P$ supported on $I$. A basic observation is that order ideals are stable under intervals: if $x,z$ belong to $I$ and $x\le_{\mathcal P} z$, then
\begin{align*}
[x,z]_{\mathcal P}\subseteq I.
\end{align*}
In particular, for every $x,z$ in $I$, we have the identity
\begin{equation}\label{eq:phi-ideal}
\Phi_x^{\mathcal P_{\mid I}}(z)=\Phi_x^{\mathcal P}(z).
\end{equation}
Namely, the hockey-stick functions of the restricted poset coincide with the original ones, as long as both arguments remain inside $I$. Nevertheless, hockey-stick domination does not, in general, pass to restrictions, since the inequalities on $\mathcal P$ may also involve mass outside $I$. The next proposition gives a simple sufficient condition under which domination is preserved.

\begin{proposition}\label{p:ideal}
Let $(\mathcal P,\le_{\mathcal P})$ be a finite poset, let $I\subseteq\mathcal P$ be an order ideal, and let $\mu,\nu\in\mathcal M_1(\mathcal P)$.
Assume that $\mu\preceq_{HS}\nu$ on $\mathcal P$ and 
\begin{align*}
\mu[z]=\nu[z],\qquad z\in \mathcal P\setminus I.
\end{align*}
Then $\mu_{\mid I}\preceq_{HS}\nu_{\mid I}$ on $\mathcal P_{\mid I}$. In particular, if $\mu[I]=\nu[I]>0$, then
$\mu^{I}\preceq_{HS}\nu^{I}$ on $\mathcal P_{\mid I}$.
\end{proposition}

\begin{proof}
Fix $x$ in $I$. Using \eqref{eq:phi-ideal} and the definition of $\mu_{\mid I}$ and $\nu_{\mid I}$, we have
\begin{align*}
\langle \nu_{\mid I}-\mu_{\mid I},\Phi_x^{\mathcal P_{\mid I}}\rangle
&=
\sum_{z\in I}\Phi_x^{\mathcal P_{\mid I}}(z)\,\bigl(\nu_{\mid I}[z]-\mu_{\mid I}[z]\bigr)
=
\sum_{z\in I}\Phi_x^{\mathcal P}(z)\,\bigl(\nu[z]-\mu[z]\bigr).
\end{align*}
Since $\mu[z]=\nu[z]$ for all $z\notin I$, we may extend the sum to $\mathcal P$ without changing its value:
\begin{align*}
\sum_{z\in I}\Phi_x^{\mathcal P}(z)\,\bigl(\nu[z]-\mu[z]\bigr)
&=
\sum_{z\in \mathcal P}\Phi_x^{\mathcal P}(z)\,\bigl(\nu[z]-\mu[z]\bigr)=
\langle \nu-\mu,\Phi_x^{\mathcal P}\rangle.
\end{align*}
By the hypothesis $\mu\preceq_{HS}\nu$ on $\mathcal P$, the last quantity is nonnegative for every $x\in\mathcal P$,
and in particular for every $x$ in $I$. This proves $\mu_{\mid I}\preceq_{HS}\nu_{\mid I}$ on $\mathcal P_{\mid I}$. Moreover, if $\mu[I]=\nu[I]>0$, then for every $x\in I$,
\begin{align*}
\big\langle \nu^{I}-\mu^{I},\Phi_x^{\mathcal P_{\mid I}}\big\rangle
=
\frac{1}{\mu[I]}\big\langle \nu_{\mid I}-\mu_{\mid I},\Phi_x^{\mathcal P_{\mid I}}\big\rangle
\ge 0,
\end{align*}
which yields $\mu^{I}\preceq_{HS}\nu^{I}$.
\end{proof}

\subsubsection{Young diagrams}

The combination of chains with direct products and ideal restrictions yields posets compatible with hockey-stick domination. Starting from measures on finite chains, products generate rectangular grids, and ideal restrictions produce Young diagram shapes. Every Young diagram poset associated with a partition
$\lambda=(\lambda_1\ge\cdots\ge\lambda_r)$ admits the representation
\begin{align*}
I_\lambda
\;\subseteq\;
C_r \times C_{\lambda_1},
\qquad
I_\lambda
=
\{(i,j)\in C_r\times C_{\lambda_1} : j\le \lambda_i\},
\end{align*}
that is, $I_\lambda$ is an order ideal of a grid.

\medskip

\noindent
By the compatibility of hockey-stick domination under direct products, and its preservation under ideal restrictions whenever the measures coincide outside the ideal, dominated measures constructed on the ambient grid induce hockey-stick dominated restrictions on the corresponding Young diagram. In particular, whenever the common mass assigned to $I_\lambda$ is positive, the conditional restrictions give hockey-stick dominated probability measures on the Young diagram.

\medskip

\noindent
In general, the main takeaway from the preceding results is that they allow the problem of establishing hockey-stick domination on a finite poset to be reduced to the corresponding problem on smaller, and typically simpler, posets through suitable poset operations.

%==============================================================
\section{Convex domination}\label{sec:6}

\noindent
In the previous sections we introduced a notion of domination tailored to our framework, namely hockey-stick domination, which arises naturally in connection with Theorem~\ref{thm:main}. One may then ask how this order is related to a larger class of test functions.\\

\noindent
A natural candidate is M\"obius convex domination. This choice is motivated by Proposition~\ref{p0}~$ii)$, which shows that every function $f$ in $\mathfrak{F}$ can be written as the difference of two nonnegative increasing M\"obius convex functions.\\

\begin{definition}
We say that $\mu$ in $\mathcal{M}(\mathcal{P})$ is dominated by $\nu$ in $\mathcal{M}(\mathcal{P})$ in the M\"obius convex order, written $\mu \preceq_{\mathrm{conv}} \nu$, if for every M\"obius convex function $f \in \mathfrak{F}$ it holds that
\begin{align*}
\langle \nu-\mu,f\rangle \ge 0.
\end{align*}
\end{definition}

\noindent 
First, we observe from Proposition~\ref{p0}~$i)$ that hockey-stick functions are M\"obius convex. Consequently, M\"obius convex domination implies hockey-stick domination. In this sense, the new order is obtained by enlarging the class of test functions considered so far. Moreover, because M\"obius convex domination is characterized through a stronger class of test functions, it may induce a finer notion of domination, capable of capturing subtler differences between the objects under consideration.\\

\begin{remark}
Lemma~\ref{l1}~$v)$ shows that M\"obius convexity is rather rigid when combined with the opposite sign condition on the second M\"obius derivative. Indeed, if $f$ is M\"obius convex and $D^2[f]\leq 0$, then $D[f]$ is constant and nonpositive on each connected component of the Hasse diagram of $\mathcal P$. Thus, on each connected component, $f$ is M\"obius linear. In this sense, the only M\"obius convex functions with nonpositive second M\"obius derivative are linear-type functions. This observation explains why M\"obius convex domination imposes, in addition to hockey-stick inequalities, moment-type constraints associated with M\"obius linear functions.
\end{remark}

\noindent
The previous observation motivates the following proposition.

\begin{proposition} \label{p2}
Let $\mu,\nu\in \mathcal M(\mathcal P)$, and assume that $\mathcal P$ has a least element $0_{\mathcal P}$. If $\mu \preceq_{\mathrm{conv}} \nu,$ then
\begin{align*}
\mu \preceq_{\mathrm{HS}} \nu
\quad \text{and} \quad
m_1[\mu]=m_1[\nu].
\end{align*}
\end{proposition}

\begin{proof}
Since every hockey-stick function is M\"obius convex, the implication
\begin{align*}
\mu \preceq_{\mathrm{conv}} \nu
\quad \Longrightarrow \quad
\mu \preceq_{\mathrm{HS}} \nu
\end{align*}
follows directly from the definitions. It remains to prove the equality of first moments.\\

\noindent Because $\mathcal P$ has a least element, we have
\begin{align*}
\phi_1^{\mathcal P}(x)
=
I[\mathbf 1](x)
=
|(-\infty,x]|
=
|[0_{\mathcal P},x]|
=
\Phi_{0_{\mathcal P}}^{\mathcal P}(x).
\end{align*}
Moreover, $\Phi_{0_{\mathcal P}}^{\mathcal P}$ is M\"obius linear, since
\begin{align*}
D[\Phi_{0_{\mathcal P}}^{\mathcal P}]=\mathbf 1.
\end{align*}
Therefore both $\Phi_{0_{\mathcal P}}^{\mathcal P}$ and $-\Phi_{0_{\mathcal P}}^{\mathcal P}$ are M\"obius convex. Applying M\"obius convex domination to these two functions gives
\begin{align*}
\langle \nu-\mu,\Phi_{0_{\mathcal P}}^{\mathcal P}\rangle\ge0
\quad\text{and}\quad
\langle \nu-\mu,-\Phi_{0_{\mathcal P}}^{\mathcal P}\rangle\ge0.
\end{align*}
Hence
\begin{align*}
\langle \nu-\mu,\Phi_{0_{\mathcal P}}^{\mathcal P}\rangle=0.
\end{align*}
Since $\phi_1^{\mathcal P}=\Phi_{0_{\mathcal P}}^{\mathcal P}$, this is exactly
\begin{align*}
m_1[\nu]-m_1[\mu]=0.
\end{align*}
This proves the result.
\end{proof}

\noindent
As a straightforward consequence, a necessary condition for M\"obius convex domination between 
\(\mu,\nu\) in $\mathcal M(\mathcal P)$, when \(\mathcal P\) has a least element \(0_{\mathcal P}\), is that equality holds for the hockey-stick test associated with \(0_{\mathcal P}\); namely,
\begin{align*}
\langle \nu-\mu,\Phi_{0_{\mathcal P}}^{\mathcal P}\rangle =0.
\end{align*}
This shows that on posets with a least element, a refinement of hockey-stick domination is necessary to ensure M\"obius convex domination. This naturally raises the question of when such a refinement is also sufficient. An instance of this phenomenon is illustrated in the following subsection.

\subsection{Characterization of convex domination on chains} \label{scr}

The chain case is particularly important because it reveals the precise relation between hockey-stick domination and classical convex comparison principles. The following result shows that, on finite chains, convex domination is completely determined by hockey-stick domination together with equality of first moments.

\begin{proposition}
For $\mu,\nu \in \mathcal{M}(C_n)$, one has
\[
\mu \preceq_{\mathrm{conv}} \nu
\]
if and only if
\[
\mu \preceq_{\mathrm{HS}} \nu
\qquad\text{and}\qquad
m_1[\mu]=m_1[\nu].
\]
\end{proposition}

\begin{proof}
The first implication follows directly from Proposition~\ref{p2}. Conversely, assume that
\[
\mu \preceq_{\mathrm{HS}} \nu
\qquad\text{and}\qquad
m_1[\mu]=m_1[\nu].
\]
Let $f$ be M\"obius convex on $C_n$. Since $C_n$ is a chain,
\[
D^2[f](i)=D[f](i)-D[f](i-1)\ge0,
\qquad i=2,\dots,n.
\]
Moreover, Proposition~\ref{p0}~ii) gives the representation
\[
f
=
\sum_{i=1}^n D^2[f](i)\,\Phi_i^{C_n}.
\]
Therefore,
\[
\langle \nu-\mu,f\rangle
=
D^2[f](1)\,\langle \nu-\mu,\Phi_1^{C_n}\rangle
+
\sum_{i=2}^n
D^2[f](i)\,\langle \nu-\mu,\Phi_i^{C_n}\rangle.
\]
Since $\Phi_1^{C_n}=\phi_1^{C_n}$, the condition
\[
m_1[\mu]=m_1[\nu]
\]
implies
\[
\langle \nu-\mu,\Phi_1^{C_n}\rangle=0.
\]
On the other hand, hockey-stick domination yields
\[
\langle \nu-\mu,\Phi_i^{C_n}\rangle\ge0,
\qquad i=1,\dots,n.
\]
Because $D^2[f](i)\ge0$ for $i=2,\dots,n$, we conclude that
\[
\langle \nu-\mu,f\rangle\ge0.
\]
Hence,
\[
\mu \preceq_{\mathrm{conv}} \nu.
\qedhere
\]
\end{proof}

The previous proposition shows that hockey-stick domination forms the core comparison structure underlying convex domination on finite chains. In particular, convex domination is recovered exactly by imposing a first-moment constraint. Thus, the chain case recovers the discrete analog of the classical characterization of convex order on $\mathbb R$.

\subsection{Constructing M\"obius convexity}

In light of the preceding results concerning the class of M\"obius convex functions, it is natural to examine this notion more closely. In particular, we investigate how M\"obius convexity behaves under the poset operations introduced above.

\begin{proposition}\label{prop:constructing-mobius-convexity}
The following properties hold.
\begin{enumerate}
\item[$i)$] For $i=1,\dots,n$, let $(\mathcal{P}_i,\le_i)$ be finite posets and let 
$f_i:\mathcal{P}_i \to \mathbb{R}$ be M\"obius convex functions such that 
\[
D_{\mathcal{P}_i}[f_i](x)\ge 0
\]
for every $x \in min(\mathcal P_i)$. Define $f:\mathcal{P}\to\mathbb{R}$ by
\[
f(x_1,\dots,x_n):=\prod_{i=1}^n f_i(x_i),
\qquad
\mathcal{P}:=\prod_{i=1}^n \mathcal{P}_i .
\]
Then $f$ is M\"obius convex on the product poset $\mathcal{P}$. Moreover, $D_{\mathcal{P}}[f]\ge 0.$

\item[$ii)$] Let $(\mathcal{P}_i,\le_i)$, $i=1,\dots,n$, be finite posets and let
$f_i:\mathcal{P}_i\to\mathbb{R}$ be M\"obius convex functions. Define $f:\mathcal{P}\to\mathbb{R}$ by
\[
f(x):=f_i(x)\quad \text{whenever } x\in \mathcal{P}_i,
\qquad
\mathcal{P}:=\bigsqcup_{i=1}^n \mathcal{P}_i .
\]
Then $f$ is M\"obius convex on the disjoint union poset $\mathcal{P}$.

\item[$iii)$] Let $(\mathcal{P},\le_{\mathcal{P}})$ and $(\mathcal{Q},\le_{\mathcal{Q}})$ be finite posets such that $\mathcal{P}$ has a greatest element $1_{\mathcal{P}}$ and $\mathcal{Q}$ has a least element $0_{\mathcal{Q}}$.  
Let $f_{\mathcal{P}}:\mathcal{P}\to\mathbb{R}$ and $f_{\mathcal{Q}}:\mathcal{Q}\to\mathbb{R}$ be M\"obius convex functions satisfying
\[
0\leq f_{\mathcal{P}}(1_{\mathcal{P}}) \quad \text{and} \quad D_{\mathcal{P}}[f_{\mathcal{P}}](1_{\mathcal{P}})
\le
f_{\mathcal{Q}}(0_{\mathcal{Q}})-f_{\mathcal{P}}(1_{\mathcal{P}}).
\]
Define $f:\mathcal{P}\oplus\mathcal{Q}\to\mathbb{R}$ by
\[
f(x):=
\begin{cases}
f_{\mathcal{P}}(x), & x\in\mathcal{P},\\
f_{\mathcal{Q}}(x), & x\in\mathcal{Q}.
\end{cases}
\]
Then $f$ is M\"obius convex on $\mathcal{P}\oplus\mathcal{Q}$.

\item[$iv)$] Let $(\mathcal{P},\le_{\mathcal{P}})$ be a finite poset and let $I\subseteq\mathcal{P}$ be an order ideal. If $f:\mathcal{P}\to\mathbb{R}$ is M\"obius convex on $\mathcal{P}$, then the restriction $f_I$ is M\"obius convex on $I$.
\end{enumerate}
\end{proposition}

\begin{proof}
The proofs amount to tracking the behavior of the M\"obius function under the poset operations considered above.

\medskip

\noindent
$i)$ For the product poset, one has
\[
\mu_{\prod_{i=1}^n \mathcal{P}_i}
\big((x_1,\dots,x_n),(y_1,\dots,y_n)\big)
=
\prod_{i=1}^n \mu_{\mathcal{P}_i}(x_i,y_i),
\]
whenever $(x_1,\dots,x_n)\le (y_1,\dots,y_n)$. Therefore,
\[
D_{\mathcal{P}}[f](x_1,\dots,x_n)
=
\prod_{i=1}^n D_{\mathcal{P}_i}[f_i](x_i).
\]
Since each $f_i$ is M\"obius convex, $D_{\mathcal{P}_i}[f_i]$ is increasing. Moreover, by the assumed nonnegativity on minimal elements, each $D_{\mathcal{P}_i}[f_i]$ is nonnegative on $\mathcal{P}_i$. Hence the product above is nonnegative and increasing on $\mathcal P$. Consequently, $f$ is M\"obius convex and $D_{\mathcal P}[f]\ge0$.

\medskip

\noindent
$ii)$ For the disjoint union, if $x,y\in\mathcal P_j$ and $x\le y$, then
\[
\mu_{\bigsqcup_{i=1}^n \mathcal{P}_i}(x,y)
=
\mu_{\mathcal{P}_j}(x,y).
\]
Therefore,
\[
D_{\bigsqcup_{i=1}^n \mathcal{P}_i}[f](x)
=
D_{\mathcal{P}_j}[f_j](x),
\]
whenever $x\in\mathcal P_j$. Since there are no comparable elements belonging to different components, the M\"obius convexity of each $f_j$ implies the M\"obius convexity of $f$ on the disjoint union.

\medskip

\noindent
$iii)$ For the ordinal sum, the M\"obius function is given by
\[
\mu_{\mathcal{P}\oplus\mathcal{Q}}(x,y)=
\begin{cases}
\mu_{\mathcal{P}}(x,y), & x,y\in\mathcal{P},\\
\mu_{\mathcal{Q}}(x,y), & x,y\in\mathcal{Q},\\
-\delta_{x,1_{\mathcal{P}}}\delta_{y,0_{\mathcal{Q}}}, & x\in\mathcal{P},\ y\in\mathcal{Q}.
\end{cases}
\]
Consequently,
\[
D_{\mathcal{P}\oplus\mathcal{Q}}[f](x)
=
D_{\mathcal{P}}[f_{\mathcal{P}}](x),
\qquad x\in\mathcal P,
\]
whereas, for $x\in\mathcal Q$,
\[
D_{\mathcal{P}\oplus\mathcal{Q}}[f](x)=
\begin{cases}
f_{\mathcal{Q}}(0_{\mathcal{Q}})-f_{\mathcal{P}}(1_{\mathcal{P}}), & x=0_{\mathcal{Q}},\\
D_{\mathcal{Q}}[f_{\mathcal{Q}}](x), & x\neq 0_{\mathcal{Q}}.
\end{cases}
\]
Since $f_{\mathcal P}$ and $f_{\mathcal Q}$ are M\"obius convex, the only comparison that is not internal to $\mathcal P$ or $\mathcal Q$ is the transition from $\mathcal P$ to $\mathcal Q$. This is guaranteed by the conditions 
\[
0 \leq f_{\mathcal{P}}(1_{\mathcal{P}}) \quad \text{and} \quad D_{\mathcal{P}}[f_{\mathcal{P}}](1_{\mathcal{P}})
\le
f_{\mathcal{Q}}(0_{\mathcal{Q}})-f_{\mathcal{P}}(1_{\mathcal{P}}).
\]
Thus $D_{\mathcal{P}\oplus\mathcal{Q}}[f]$ is increasing, and hence $f$ is M\"obius convex.

\medskip

\noindent
$iv)$ Finally, let $I$ be an order ideal. If $x,y\in I$ and $x\le y$, then the interval $[x,y]_{\mathcal P}$ is contained in $I$. Therefore,
\[
\mu_{\mathcal P}(x,y)=\mu_I(x,y),
\]
and hence
\[
D_{\mathcal P}[f](x)=D_I[f_I](x),
\]
for every $x\in I$. Since $f$ is M\"obius convex on $\mathcal P$, the restriction $D_I[f_I]$ is increasing on $I$. Therefore $f_I$ is M\"obius convex.
\end{proof}

\begin{remark}
Observe that, by Lemma~\ref{l1}~$iv)$, the first statement concerns functions that are M\"obius  convex, increasing and nonnegative. The main takeaway of the previous result is that it provides a way of lifting M\"obius convex functions from the underlying posets to the posets generated by the operations above. In particular, it allows us to construct M\"obius convex functions on product posets, disjoint unions, ordinal sums, and order ideals in terms of M\"obius convex functions defined on the smaller components.
\end{remark}

%==============================================================

\section{Zolotarev distance}\label{sec:zolotarev}

Motivated by the real metric of Zolotarev type, we measure discrepancies between probability measures by testing against functions whose derivatives up to order two are suitably controlled. The material in this section provides the framework underlying Theorem~\ref{thm:main}, which was stated earlier in the paper in order to anticipate the main comparison result.\\

\noindent For $\mu,\nu$ in $\mathcal M_1(\mathcal P)$, we set $\rho:=\nu-\mu$. Define
\begin{equation}\label{eq:dzeta-def}
d_{\zeta}(\mu,\nu)
:=
\sup\Big\{
\bigl|\langle \rho,f\rangle\bigr|
\,:\,
f:\mathcal P\to\mathbb R,\ \ \|D^2[f]\|_{\infty}\le 1
\Big\}.
\end{equation}

\noindent Since $D$ and $I$ are defined purely from the order on $\mathcal P$, the distance $d_{\zeta}$ is invariant under poset isomorphisms. Namely, if $T:\mathcal P\to\mathcal Q$ is a poset isomorphism, then, for any $\mu,\nu$ in $\mathcal M_1(\mathcal P)$,
\[
d_{\zeta}(\mu,\nu)=d_{\zeta}(T_{\#}\mu,T_{\#}\nu).
\]
The finiteness of the poset under consideration further allows us to compare $d_{\zeta}$ with the total variation distance, as illustrated in the following.

\begin{lemma}\label{lem:mobius-chain-bound}
Let $\mathcal P$ be a finite poset with $|\mathcal P|=n \geq 2$, and let $\mu_{\mathcal P}$ denote its M\"obius function. Then
\[
\|\mu_{\mathcal P}\|_\infty
:=
\max\bigl\{|\mu_{\mathcal P}(x,y)|:\ x\le_{\mathcal P} y\bigr\}
\le 2^{\,n-2}.
\]
\end{lemma}

\begin{proof}
Let $\zeta$ denote the zeta function of $\mathcal P$ and $\delta$ the identity of the incidence algebra. Since $\mu_{\mathcal P}=\zeta^{-1}$, we may write the expansion
\[
\mu_{\mathcal P}
=
\zeta^{-1}
=
\sum_{k\ge 0}(-1)^k(\zeta-\delta)^k,
\]
which is finite because $\mathcal P$ is finite. Fix $x< y$ in $\mathcal P$. The coefficient of $(\zeta-\delta)^k$ at $(x,y)$ counts the number of strict chains
\[
x=x_0<x_1<\cdots<x_k=y
\]
of length $k$ from $x$ to $y$. Denoting this number by $c_k(x,y)$, we obtain
\[
\mu_{\mathcal P}(x,y)
=
\sum_{k\ge 0}(-1)^k\,c_k(x,y).
\]
Consequently,
\[
|\mu_{\mathcal P}(x,y)|
\le
\sum_{k\ge 0} c_k(x,y) \le
\sum_{k\ge 1}^{n-1}\binom{n-2}{k-1}\leq 2^{n-2}.
\]
If $x=y$, then $\mu_{\mathcal P}(x,x)=1\le 2^{n-2}$, since $n\ge2$. Taking the supremum over all $x\le y$ yields the claim.
\end{proof}

\noindent In the sequel, $d_{TV}$ will denote total variation distance.

\begin{proposition}\label{prop:tv-vs-dzeta}
Given $|\mathcal P|\ge2$ and $\mu,\nu\in\mathcal M_1(\mathcal P)$, it holds that
\[
d_{TV}(\mu,\nu)
\le |\mathcal P|^3\|\mu_{\mathcal P}\|_{\infty}^2\, d_{\zeta}(\mu,\nu)
\le |\mathcal P|^3 4^{|\mathcal P|-2} d_{\zeta}(\mu,\nu).
\]
Here
\[
\|\mu_{\mathcal P}\|_{\infty}
:=
\max\bigl\{\,|\mu_{\mathcal P}(x,y)|:\ x\le_{\mathcal P} y\,\bigr\}.
\]
\end{proposition}

\begin{proof}
Fix $A\subseteq\mathcal P$ and consider $f=\mathbbm 1_A$. As a consequence of Proposition~\ref{p0}~$ii)$,
\[
f(x)=\sum_{z\in\mathcal P} \Phi_z^{\mathcal{P}}(x)\, D^2[f](z),
\]
for $x$ in $\mathcal P$. Integrating against $\nu-\mu$, we get 
\[
|\nu[A]-\mu[A]|
=
\bigl|\langle \nu-\mu,f\rangle\bigr|
\le \sum_{z\in\mathcal P}\bigl|\langle \nu-\mu,\Phi_z^{\mathcal{P}}\rangle\bigr|\, |D^2[f](z)|.
\]
Now, by Proposition~\ref{p0}~$i)$,
\[
D^2[\Phi_z^{\mathcal P}]=\delta_z,
\]
and therefore
\[
\big\|D^2[\Phi_z^{\mathcal P}]\big\|_\infty\le 1.
\]
Consequently,
\[
\bigl|\langle \nu-\mu,\Phi_z^{\mathcal{P}}\rangle\bigr|
\le d_\zeta(\mu,\nu),
\]
for all $z$ in $\mathcal P$. We then conclude that 
\[
|\nu[A]-\mu[A]|
\le d_\zeta(\mu,\nu)\sum_{z\in\mathcal P}|D^2[f](z)|.
\]
By the definition of $D$, we have
\[
|D^2[f](z)|
\le
\sum_{x\le z}\sum_{y\le x}
|f(y)|\,|\mu_{\mathcal P}(y,x)|\,|\mu_{\mathcal P}(x,z)|
\le
\|\mu_{\mathcal P}\|_{\infty}^2\, I^2[\mathbf{1}](z)
\le
\|\mu_{\mathcal P}\|_{\infty}^2\,|\mathcal P|^2.
\]
Consequently,
\[
\sum_{z\in\mathcal P}|D^2[f](z)|
\le
\|\mu_{\mathcal P}\|_{\infty}^2\,|\mathcal P|^3.
\]
Thus,
\[
|\nu[A]-\mu[A]|
\le
\|\mu_{\mathcal P}\|_{\infty}^2\,|\mathcal P|^3\, d_{\zeta}(\mu,\nu).
\]
Taking the supremum over $A\subseteq\mathcal P$ gives the first bound.  The second bound follows from Lemma~\ref{lem:mobius-chain-bound}.
\end{proof}

\begin{remark}
The bound above is quite coarse as a function of $|\mathcal P|$. In applications where finer control of $I^2[\mathbf{1}]$ is available, the constant can most likely be improved.
\end{remark}

%==============================================================

\section{Proof of Theorem~\ref{thm:main}}\label{sec:8}
In this section we prove the main theorem of the paper that we state again for the convenience of the reader.
\begin{theorem}
Let $\mu,\nu$ be probability measures on $\mathcal P$. If $\mu\preceq_{HS}\nu$, then
\begin{align*}
d_{\zeta}(\mu,\nu)= \frac12\, m_2[\nu-\mu].
\end{align*}
\end{theorem} 
\begin{proof}
    
 Let $f:\mathcal P\to\mathbb R$ be a function satisfying $\|D^2[f]\|_\infty\le 1$, and let $\mu,\nu\in\mathcal M_1(\mathcal P)$ such that $\mu\preceq_{HS}\nu$. According to Proposition \ref{p0} $ii)$
we have
\begin{align*}
\langle\nu-\mu,f\rangle
&=
\sum_{z\in\mathcal P}
D^2[f](z)\,
\langle \nu-\mu,\Phi_z^{\mathcal P}\rangle .
\end{align*}
Therefore,
\[
\bigl|\langle\nu-\mu,f\rangle\bigr|
\le
\|D^2[f]\|_\infty
\sum_{z\in\mathcal P}
\bigl|\langle \nu-\mu,\Phi_z^{\mathcal P}\rangle\bigr|
.
\]
Since $\|D^2[f]\|_\infty\le 1$ and $\mu\preceq_{HS}\nu$ it follows that
\[
\bigl|\langle\nu-\mu,f\rangle\bigr|
\le \sum_{z\in\mathcal P}
\bigl|\langle \nu-\mu,\Phi_z^{\mathcal P}\rangle{} \bigr|=
\sum_{z\in\mathcal P}
\langle \nu-\mu,\Phi_z^{\mathcal P}\rangle.
\]
\noindent As a consequence,
\[
d_\zeta(\mu,\nu)
\le
\sum_{z\in\mathcal P}
\langle \nu-\mu,\Phi_z^{\mathcal P}\rangle
.
\] On the other hand,  since $D^2[\phi_{2}^{\mathcal P}]=2$, we also have  according to Proposition \ref{p0} $ii)$

\[\langle \nu-\mu,\phi_{2}^{\mathcal P}\rangle
=
2\sum_{z\in\mathcal P}
\langle \nu-\mu,\Phi_z^{\mathcal P}\rangle,
\]  and therefore  since $\mu\preceq_{HS}\nu$
\[
d_\zeta(\mu,\nu)
\ge
\bigl|\langle\nu-\mu,\frac{\phi_{2}^{\mathcal P}}{2}\rangle\bigr|
=
\sum_{z\in\mathcal P}
\langle \nu-\mu,\Phi_z^{\mathcal P}\rangle
.
\]
This yields the identity 
\[
d_\zeta(\mu,\nu)
=
\sum_{z\in\mathcal P}
\langle \nu-\mu,\Phi_z^{\mathcal P}\rangle
.
\]  Finally, by noting that,
\[
m_2[\nu]-m_2[\mu]
=
\langle \nu-\mu,\phi_{2}^{\mathcal P}\rangle
=
2\sum_{z\in\mathcal P}
\langle \nu-\mu,\Phi_z^{\mathcal P}\rangle.
\] Combining the previous identities gives
\[
d_\zeta(\mu,\nu)
=
\frac12\bigl(m_2[\nu]-m_2[\mu]\bigr)
=
\frac12\,m_2[\nu-\mu],
\]
which proves the theorem.\\
\end{proof}

\noindent \textbf{Acknowledgements}\\
Arturo Jaramillo Gil was supported by
the grant CBF2023-2024-2088. 
Saylé Sigarreta was supported by CONAHCYT 2023-2024 project CBF2023-2024-1842.

\bibliographystyle{plain} % Choose your preferred style (plain, alpha, etc.)
\bibliography{Bib}

\begin{thebibliography}{10}

\bibitem{MR1909919}
Michael~V. Boutsikas and Eutichia Vaggelatou.
\newblock On the distance between convex-ordered random variables, with applications.
\newblock {\em Adv. in Appl. Probab.}, 34(2):349--374, 2002.

\bibitem{boyd2004convex}
Stephen Boyd and Lieven Vandenberghe.
\newblock {\em Convex Optimization}.
\newblock Cambridge University Press, Cambridge, 2004.

\bibitem{murota2003}
Kazuo Murota.
\newblock {\em Discrete Convex Analysis}.
\newblock Society for Industrial and Applied Mathematics (SIAM), Philadelphia, PA, 2003.

\bibitem{niculescu2006convex}
Constantin~P. Niculescu and Lars-Erik Persson.
\newblock {\em Convex Functions and Their Applications: A Contemporary Approach}.
\newblock CMS Books in Mathematics. Springer, New York, 2006.

\bibitem{NourdinPeccati2009}
Ivan Nourdin and Giovanni Peccati.
\newblock Stein's method on wiener chaos.
\newblock {\em Probability Theory and Related Fields}, 145(1-2):75--118, 2009.

\bibitem{n11}
Gian-Carlo Rota.
\newblock On the foundations of combinatorial theory: I. theory of m{\"o}bius functions.
\newblock In {\em Classic Papers in Combinatorics}, pages 332--360. Springer, 1964.

\bibitem{shaked2007stochastic}
Moshe Shaked and J.~George Shanthikumar.
\newblock {\em Stochastic Orders}.
\newblock Springer, New York, 2007.

\bibitem{stanley2011enumerative}
Richard~P. Stanley.
\newblock {\em Enumerative Combinatorics, Volume 1}, volume~49 of {\em Cambridge Studies in Advanced Mathematics}.
\newblock Cambridge University Press, Cambridge, second edition, 2011.

\bibitem{topkis1998}
Donald~M. Topkis.
\newblock {\em Supermodularity and Complementarity}.
\newblock Princeton University Press, Princeton, NJ, 1998.

\bibitem{trotter1992combinatorics}
William~T. Trotter.
\newblock {\em Combinatorics and Partially Ordered Sets: Dimension Theory}.
\newblock Johns Hopkins University Press, Baltimore, 1992.

\end{thebibliography}

\end{document}